\documentclass{article}
\usepackage{amsmath,
amssymb,
amsthm,
amsfonts,
latexsym,
hyperref,
enumerate,
color
}

\usepackage{graphicx} 

\def\F{{\mathbb F}}
\def\S{\mathcal{S}}

\def\D{\mathcal{D}}
\def\PG{\mbox{\rm PG}}

\newtheorem{lemma}{Lemma}[section]
\newtheorem{theorem}[lemma]{Theorem}
\newtheorem{corollary}[lemma]{Corollary}

\newtheorem{remark}[lemma]{Remark}
\newtheorem{result}[lemma]{Result}

\title{A note on strong blocking sets \\and higgledy-piggledy sets of lines}
\author{Stefano Lia \and Geertrui Van de Voorde}
\date{}

\begin{document}

\maketitle

\begin{abstract} 
This paper studies {\em strong blocking sets} in the $N$-dimensional finite projective space $\PG(N,q)$. 
We first show that certain unions of blocking sets cannot form strong blocking sets, which leads to a new lower bound on the size of a strong blocking set in $\PG(N,q)$.
 Our second main result shows that, for $q>\frac{2}{ln(2)}(N+1)$, there exists a subset of $2N-2$ lines of a Desarguesian line spread in $\PG(N,q)$, $N$ odd, in {\em higgledy-piggledy arrangement}; thus giving rise to a strong blocking set of size $(2N-2)(q+1)$.

\end{abstract}

\section{Introduction}
\subsection{Strong blocking sets and linear codes}\label{subs}
We start by introducing the necessary  definitions for blocking sets and linear codes. A {\em$t$-fold  blocking set} in a projective space $\PG(N,q)$, $q=p^h$, $p$ prime, is a set of points $\mathcal{B}$ such that every hyperplane contains at least $t$ points of $\mathcal{B}$. A $1$-fold blocking set is simply called a {\em blocking set}. Note that in this paper, we only consider blocking sets with respect to hyperplanes. A blocking set $\mathcal{B}$ is called {\em minimal} if for every point $P\in \mathcal{B}$, the set $\mathcal{B}\setminus\{P\}$ is not a blocking set; that is,  every point of $\mathcal{B}$ lies on at least one tangent hyperplane to the set $\mathcal{B}$. A blocking set in $\PG(N,q)$ is called {\em strong} if it meets every hyperplane in a set of points, spanning that hyperplane. Depending on the context, strong blocking sets have also been called {\em cutting} blocking sets \cite{bartoli} or {\em generator sets} \cite{FancsaliSziklai2014} in the literature. It follows from the definition that a strong blocking set in $\PG(N,q)$ is an $N$-fold blocking set; but the converse is not necessary true (unless $N=2$).

The study of $1$-fold blocking sets forms a classical problem within finite geometry, and the connection with problems from coding theory (e.g. the study of codewords of small weight in the code of points and lines of a projective plane) has formed an extra motivation for their study. For a survey about $1$-fold blocking sets, see \cite{survey}. We will now see that strong blocking sets too are connected to coding theory.

A linear {\em code} $C$ over $\mathbb{F}_q$ of {\em length} $n$ and {\em dimension} $k$ is a $k$-dimensional vector subspace of $\mathbb{F}_q^n$. The vectors of $C$ are called {\em codewords}. A code is called {\em minimal} if the {\em support} of any codeword $v$ never properly contains the support of a linearly independent codeword. The support of a codeword is the set of positions in which this codeword is non-zero.

Strong blocking sets were first introduced in \cite{DavydovGiuliettiMarcuginiPambianco2011} in relation to covering codes and saturating sets. 
A {\em $\rho$-covering code} $C\leq \mathbb{F}_q^n$ is a code with the property that every vector in $\mathbb{F}_q^n$ has distance at most $\rho$ from a code word of $C$, and $\rho$ is the least integer with such property. The integer $\rho$ is the {\em covering radius} of $C$.
A {\em $\rho$-saturating set} in $\PG(N,q)$ is a set $\mathcal{S}$ of points such that every point $P$ of $\PG(N,q)$ lies on a subspace spanned by at most $\rho+1$ points of $\mathcal{S}$. This notion relates to classical problems in algebraic geometry about secant varieties: a $\rho$-saturating set $\mathcal{S}$ can be defined as a set such that the $\rho$-th secant variety of any variety meeting $\PG(N,q)$ in $\mathcal{S}$ contains each point of $\PG(N,q)$.

A strong blocking set in $\PG(N,q)$ forms, by embedding $\PG(N, q)$ in $\PG(N, q^N )$, an $(N-1)$-saturating set of $\PG(N, q^N )$, see \cite[Theorem 3.2]{DavydovGiuliettiMarcuginiPambianco2011}. Therefore, an upper bound on the minimum size of a strong blocking set leads to a corresponding bound on the minimum size of a saturating set.

In this area, typical problems are to find small upper bounds for the minimum size of a covering code of given covering radius and dimension and to find families of codes with good asymptotic covering density.
Strong blocking set can be used to deal with these problems, see for example \cite{lins,DavydovGiuliettiMarcuginiPambianco2011}.

More recently, strong blocking sets gained further interest for the coding theory community as they were proved to be in one to one correspondence with minimal codes, \cite{AlfaranoBorelloNeri2022,alfarano,TangQiuLiaoZhou2021}; it has been shown that a minimal code of length $n$ and dimension $k$ corresponds to a strong blocking set containing $n$ points in $\PG(k-1,q)$. Hence, a lower bound on the size of a strong blocking set in $\PG(k-1,q)$ leads to a lower bound on the length of a minimal $[n,k]$-code. We will use this equivalence to translate Theorem \ref{main1} into the coding theoretical language (see Corollary \ref{cor}.)


\subsection{Strong blocking sets and higgledy-piggledy sets}

One way of constructing strong blocking sets is to take the union of subspaces in {\em higgledy-piggledy} arrangment; where the definition of a set $S$ of subspaces in higgledy-piggledy arrangment is precisely that the points contained in at least one of the elements of $S$ form a strong blocking set. 


The following results are known for higgledy piggledy line sets in $\PG(N,q)$, $N$ odd.
\begin{result}\label{generalbound}\begin{itemize}\item[(i)] \cite{heger22} A higgledy-piggledy set of lines in $\PG(N,q)$ contains at least $N+\lfloor \frac{N}{2}\rfloor-\lfloor \frac{N-1}{q}\rfloor$ elements.
\item[(ii)] \cite{FancsaliSziklai2014} There exists a higgledy-piggledy set of $2N-1$ lines in $\PG(N,q)$, $N$ odd, $q\geq 2N-1$.
\end{itemize}
\end{result}


We will show in Theorem \ref{main2} that for $q$ large enough, it is possible to give a construction with one line fewer than in Result \ref{generalbound} (ii). More precisely, we show that if $q>\frac{2}{ln(2)}(N+1)$ there exists a higgledy-piggledy set of lines in $\PG(N,q)$, $N$ odd, of size $2N-2$.


It is a natural idea to try and construct higgledy-piggledy sets of subspaces whose elements are contained in a {\em spread}. A spread of $\PG(N,q)$ is a set of subspaces partitioning its point set and an elementary construction via {\em field reduction} gives rise to {\em Desarguesian spreads}. We refer to \cite{fieldreduction} for more information on field reduction. 
We see that when a higgledy-piggledy set $\mathcal{L}$ is contained in a Desarguesian spread, it is the image of some point set $\mathcal{P}$ under field reduction. The property that $\mathcal{L}$ is a higgledy-piggledy set translates into a property regarding {\em $\F_q$-linear sets}  containing $\mathcal{P}$. More precisely, it was observed in \cite{lins} that a set $\S$ of points in $\PG(\frac{N-1}{2},q^2)$  which is not contained in any linear set of rank at most $N-1$ is a higgledy-piggledy of $\PG(N,q)$ under field reduction. In \cite{bartoli}, the authors find a set of $7$ lines contained in a Desarguesian spread of $\PG(5,q)$ by essentially making use of this point of view; that is, they found a suitable point set of size $7$ in $\PG(2,q^2)$. The higgledy-piggledy set found in this paper will be constructed using the elements of a Desarguesian line spread in $\PG(2n-1,q)$ rather than a point set in $\PG(n-1,q^2)$. We will use of the well-known fact that the elementwise stabiliser of a Desarguesian line spread acts transitively on the points of each spread element (see e.g. \cite{fieldreduction}). A corollary of this property is the following:
\begin{result}\label{stabiliser}  Let $\mathcal{S}$ be a subset of a Desarguesian line spread $\mathcal{D}$ in $\PG(N,q)$, $N$ odd, and let $P_1$ and $P_2$ be points lying on the same line $L$ of $\mathcal{D}$. Let $s_i$ be the number of $t$-dimensional subspaces $\pi$ through $P_i$ such that $\pi$ meets all elements of $\mathcal{S}$. Then $s_1=s_2$.
\end{result}

\begin{remark}
A related result was recently derived in \cite{anurag}, where the authors show that {\em for $N$ sufficiently large}, there exists a strong blocking set in $\PG(N,q)$ of size $\frac{2N+2}{\log_q(\frac{q^4}{q^3-q+1})}(q+1)$ by using a union of a set of lines through the origin that forms a blocking set with respect to co-dimension $2$-spaces in an affine space. Note this strong blocking set arises from a line set in an affine space which is then projected onto the projective space so it does not arise from a higgledy-piggledy set of lines.
\end{remark}

\section{Strong blocking sets arising from the union of blocking sets}
We have seen in the introduction that every strong blocking set in $\PG(N,q)$ is an $N$-fold blocking set. 
It is easy to see that a line and a Baer subplane are both blocking sets in $\PG(N,q)$. Furthermore, it is well-known that if a blocking set in $\PG(N,q)$ does not contain a line, it has size at least that of a Baer subplane \cite{beutelspacher}. For the motivation of this section, consider the case $\PG(3,q)$: in that case, every strong blocking set is a triple blocking set but, as mentioned before, the union of three disjoint blocking sets does not necessarily give rise to a strong blocking set. We also know from Theorem \ref{generalbound} that it is not possible for the union of $3$ lines (which are blocking sets) to form a strong blocking set in $\PG(3,q)$.


The following result from \cite{bartoli}, which is valid in $\PG(3,q)$ where $q$ is a third power, is noteworthy in this context. 
\begin{result}
    There exists three subgeometries $\PG(3,q^{1/3})$ in $\PG(3,q)$, $q=q_0^3$, whose union forms a strong blocking set of size $3(q+q^{2/3}+q^{1/3}+1).$
\end{result}
This shows that it is possible to construct a strong blocking set as the union of three blocking sets; but not from three lines. This motivated the question on which unions of blocking sets in $\PG(3,q)$, and more generally, $\PG(N,q)$, can give rise to strong blocking sets.
A result of Bar\'at and Storme \cite{barat} (Result \ref{thm:barat}) shows that a  $N$-fold blocking set in $\PG(N,q)$ of small size necessarily contains the union of 
$N$ disjoint lines and/or Baer subplanes. So if a strong blocking set is small, it necessarily contains the union of lines and Baer subplanes, which explains why we investigate such unions.



Result \ref{generalbound}(i) shows that a set of lines of $\PG(N,q)$, $q\geq N$, forming a strong blocking set contains at least $N+\lfloor N/2\rfloor$ lines.
We first derive a similar result for lines and Baer subplanes in Corollary \ref{cor2}.

\begin{lemma}\label{lemmaunion} Let $q\geq N^2$, $N\geq 3$. Let $\mathcal{S}=\bigcup_{i=1}^k B_i$ be a strong blocking set of $\mathrm{PG}(N,q)$, given by the union of $k$ disjoint minimal blocking sets $B_i$ that are either lines are Baer subplanes. Then no hyperplane contains all but at most $N-1$ of the sets $B_i$.
\end{lemma}
\begin{proof}
Let $H$ be an hyperplane containing all but at most $(n-1)$ of the sets $B_i$, and relabel the sets $B_i$ that are not contained in $H$ as $B_1,\ldots,B_s$, where $s\leq n-1$. Since each $B_i$ is a blocking set, we find (at least) one point, say $P_i$ contained in $B_i\cap H$, $i=1,\ldots,s$. The $s$ points $P_1,P_2,\ldots,P_s$, span at most an $(n-2)$-dimensional space of $H$ so they are contained in some $n-2$-dimensional space $H'$ which is a hyperplane of $H$. 
There are $q+1$ hyperplanes of $\PG(N,q)$ through $H'$, one of which is $H$ and contains $B_{s+1},\ldots,B_k$. Of the other $q$ hyperplanes, at most $(N-1)(\sqrt{q}+1)$ contain further points of $\mathcal{S}$: there are at most $N-1$ blocking sets $B_i$ not contained in $H$, and the point $P_i$ lies on at most $\sqrt{q}+1$ lines in $B_i$ containing further points of $B_i$.
Since $q\geq N^2$, we have that $N\leq \sqrt{q}$, hence $(N-1)(\sqrt{q}+1)\leq q-1<q$. 
It follows that at least one hyperplane $\pi$ through $H'$ does not contain any points of $\mathcal{S}$ outside of $H'$. Since $\pi$ meets $\mathcal{S}$ in a set not spanning $\pi$, $\mathcal{S}$ is not a strong blocking set.
\end{proof}


We obtain the following corollary.

\begin{corollary} \label{cor2} Let $q\geq N^2$.
The union of $i$ lines and $k$ Baer subplanes in $\PG(N,q)$, with $i\leq \lfloor\frac{N}{2}\rfloor$ and $k\leq N-1$ is not a strong blocking set.
\end{corollary}
\begin{proof}
    Let $B_1,\ldots,B_i$ be the $i$ lines and $B_{i+1},\ldots,B_{i+k}$ be the $k$ Baer subplanes. Since $i$ lines span at most a $2i-1$-dimensional space, there is a hyperplane containing $B_1,\ldots,B_i$. Since $k\leq N-1$, the result follows from Lemma \ref{lemmaunion}.
    
    \end{proof}


In the case that we only have $k=N$ sets $B_i$ making up the set $\mathcal{S}$, it is trivially true (when $N\geq 3$) that all but $N-1$ of those $N$ are contained in a hyperplane of $\PG(N,q)$.
\begin{corollary}\label{cor3} Let $q\geq N^2$. The union of $N$ lines and/or subplanes does not form a strong blocking set in $\PG(N,q)$.
\end{corollary}

In \cite{barat}, Bar\'at and Storme showed the following.
\begin{result}\label{thm:barat}
    Let $B$ be a minimal $s$-fold blocking set in $\PG(N,q)$, $q=p^h$, $p$ prime, $N\geq 3$, $q>661$, with \[
    |B|< sq +c_pq^{2/3} -(s-1)(s-2)/2,
    \]
    where $s<\min(c_pq^{1/6}, q^{1/4}/2)$, $c_2=c_3=2^{-1/3}$ and $c_p=1$ for $p>3$.
    Then $B$ contains the disjoint union of $s$ disjoint lines and/or Baer subplanes.
\end{result}

Combining the results by Barat and Storme on $s$-fold blocking sets with the previous lemma we have the following non-existence result.

\begin{theorem}
    
\label{main1} 
The size of a strong blocking set in $\PG(N,q)$, $N\geq3$, $q=p^h$, $p$ prime, $q\geq \max\{661,a_pN^6,16N^4\}$, where $a_p=4$ for $p=2,3$ and $a_p=1$ for $p>3$,

is at least $$Nq+ c_pq^{2/3} -(N-1)(N-2)/2,$$ 

\end{theorem}

\begin{proof} Assume to the contrary that $\mathcal{S}$ is a strong blocking set of size $S<Nq+ c_pq^{2/3} -(N-1)(N-2)/2$.
Then $\mathcal{S}$ is an $N$-fold blocking set. Let $B$ be the a minimal $N$-fold blocking set contained in $\mathcal{S}$. By Result \ref{thm:barat}, $B$ contains a set $B'$ which is the disjoint union of $N$ lines and/or Baer subplanes. 
From Corollary \ref{cor3}, we know that there is a hyperplane $H$ meeting $B'$ in points spanning at most an $(N-2)$-space. Let $\pi$ be that $(N-2)$-space.
There are at most $Nq+ c_pq^{2/3} -(N-1)(N-2)/2-N(q+1)<q+1$ points in $\mathcal{S}\setminus \mathcal{B'}$. Hence, at least one hyperplane through $\pi$ contains no further points of $\mathcal{S}$, a contradiction since $\mathcal{S}$ is a strong blocking set.

\end{proof}

Using the correspondence between minimal codes and strong blocking sets introduced in Subsection \ref{subs}, we obtain the following corollary. 

\begin{corollary}\label{cor}
    A minimal code of dimension $k$ over $\F_q$, $q=p^h$, $p$ prime, $q\geq \max\{661,(k-1)^2\}$, has length at least $$(k-1)q+ c_pq^{2/3} -(k-2)(k-3)/2.$$
\end{corollary}
\begin{remark} Going through the details of the paper \cite{barat}, we see that for the case $N=3$, the bound $q>661$ can be improved to $q>9$. Hence, for this particular case, we have showed that the size of a strong blocking set in $\PG(3,q)$, $q>9$, is at least $3q+c_pq^{2/3}-1$; and that a minimal code of dimension $4$ over $GF(q)$, $q>9$, has length at least $3q+c_pq^{2/3}-1$.
\end{remark}

\section{Construction of a higgledy-piggledy set of lines of size \texorpdfstring{$2N-2$} arising from a spread}
In this section, we derive our second main theorem. For ease of notation, since $N$ is odd, we will use $N=2n-1$ throughout this section. We will show that if $q\leq \frac{4n}{ln(2)}$, there exists a higgledy-piggledy line set of size $4n-4$ contained in a Desarguesian spread of $\PG(2n-1,q)$. We construct this higgledy-piggledy set in two parts: in the next lemma, we will find a suitable set of $2n-1$ lines, which we will extend in the main theorem to a higgledy-piggledy set of size $4n-4$.

\begin{lemma} \label{lemma:starting lines}

Let $q\geq 4n$, $n\geq 2$. There exists a set of $2n-1$ lines $\mathcal{L}$ of a Desarguesian spread $\mathcal{D}$ in 
$\mathrm{PG}(2n-1,q)$ such that the number of $(2n-3)$-spaces meeting all lines of $\mathcal{L}$ is at most $2n(q+1)^{2n-3}$.

\end{lemma}
\begin{proof}

We will use induction on $n$ to show that there exist:
\begin{itemize}
\item a set of $2n-2$ lines $\mathcal{L}$ of a Desarguesian spread $\mathcal{D}$ in $\PG(2n-1,q)$ such that the number of $(2n-3)$-spaces meeting all lines of $\mathcal{L}$ is at most $(n-1)(q+1)^{2n-2}$; 
\item a line $M$ of $\mathcal{D}$ such that 
\begin{itemize}
\item the number of $(2n-3)$-spaces meeting $M$ in exactly a point and meeting all lines of $\mathcal{L}$ is at most $2(n-1)(q+1)^{2n-3}$, and 
\item the number of $(2n-3)$-spaces containing $M$ and meeting all lines of $\mathcal{L}$ is at most $2(n-1)(q+1)^{2n-4}$.
\end{itemize}
\end{itemize}
Consider the base case $n=2$. Let $\mathcal{L}$ be any two distinct lines of a Desarguesian spread in $\PG(3,q)$. It is clear that the number of lines meeting both lines of $\mathcal{L}$ is precisely $(q+1)^{2}$. Furthermore, let $M$ be any line of the Desarguesian spread, not in $\mathcal{L}$. Then the lines meeting the three lines of $\mathcal{L}\cup \{M\}$ are precisely the $q+1\leq 2(q+1)$ lines of the opposite regulus determined by $\mathcal{L}$ and there is no line containing $M$ and meeting the two lines of $\mathcal{L}$.

So now assume that there is a set $\mathcal{L}$ of $2n-4$ lines of a Desarguesian spread $\mathcal{D}_\Sigma$ in $\Sigma=\PG(2n-3,q)$ such that the number of $(2n-5)$-spaces meeting all lines of $\mathcal{L}$ is at most $(n-2)(q+1)^{2n-4}$, and a line $M$, not in $\mathcal{L}$ such that the number of $(2n-5)$-spaces meeting all lines of $\mathcal{L}$ meeting $M$ in a point is at most $2(n-2)(q+1)^{2n-5}$ and the number of $(2n-5)-$spaces containing $M$ and meeting all lines of $\mathcal{L}$ is at most $2(n-2)(q+1)^{2n-6}$.

Now embded $\Sigma$ in $\PG(2n-1,q)$ and extend the Desarguesian spread $\mathcal{D}_\Sigma$ in $\Sigma$ to a Desarguesian spread  $\mathcal{D}$ in $\PG(2n-1,q)$. 
Let $M_0$ and $M_1$ be two lines of $\mathcal{D}$, not in $\Sigma$ such that the $3$-space $\mu=\langle M_0,M_1\rangle$ meets $\Sigma$ precisely in the line $M$.
We will first show that the number of $(2n-3)$-spaces meeting $\mathcal{L}\cup \{M_0,M_1\}$ is at most $(n-1)(q+1)^{2n-2}$.

Any $(2n-3)$-space $S$ meeting all lines of $\mathcal{L}\cup \{M_0,M_1\}$ meets $\Sigma$ in either a $(2n-4)$-dimensional or a $(2n-5)$-dimensional space.

First consider the $(2n-3)$-spaces meeting $\Sigma$ in a hyperplane of $\Sigma$. Note that in the case that $S$ meets $\Sigma$ in a hyperplane of $\Sigma$, it is impossible that $S$ contains $\mu$, since $M_0$ is contained in $\mu$ and disjoint from $\Sigma$.

Each of the $q^{2n-3}+q^{2n-4}$ hyperplanes $H$ of $\Sigma$ meeting $M$ in exactly a point give rise to exactly one $(2n-3)$-space meeting all lines of $\mathcal{L}\cup \{M_0,M_1\}$, namely the $(2n-3)$-space spanned by $H$ and the unique transversal line to $M_0,M_1$ through $H\cap M$.
Each of the $\frac{q^{2n-4}-1}{q-1}$ hyperplanes of $\Sigma$ through $M$ gives rise to $q+1$ $(2n-3)$-spaces meeting all lines of $\mathcal{L}\cup \{M_0,M_1\}$, namely one for each of the $q+1$ planes of $\mu$ through $M$.
This gives us $q^{2n-3}+q^{2n-4}+\frac{q^{2n-4}-1}{q-1}(q+1)$  subspaces meeting all lines of $\mathcal{L}\cap \{M_0,M_1\}$ and meeting $\Sigma$ in a hyperplane. This number is smaller than $2q^{2n-3}$ if $q\geq 3$.

Now consider spaces meeting all lines of $\mathcal{L}\cup \{M_0,M_1\}$ and meeting $\Sigma$ in a $(2n-5)-$space. We will make a distinction depending on their intersection size with the line $M$.

First consider such spaces $S$ containing the line $M$. Each of the $(2n-5)$-spaces $\tau$ meeting all lines of $\mathcal{L}$ and containing $M$ lies on at most $(q+1)(q^2+q+1)$ $(2n-3)$-spaces meeting all lines of $\mathcal{L}\cup \{M_0,M_1\}$ since each of the $q+1$ planes through $M$ in $\mu$ gives rise to a $(2n-4)$-space meeting all lines of $\mathcal{L}\cup \{M_0,M_1\}$, and each of those $(2n-4)$-spaces lies on $q^2+q+1$ $(2n-3)$-spaces. There are at most $2(n-2)(q+1)^{2n-6}$ such $(2n-5)$-spaces $\tau$.

Each of the $(2n-5)$-spaces $\tau$ meeting all lines of $\mathcal{L}$ and meeting $M$ in a point lies on at most $(q^2+q+1)$ $(2n-3)$-spaces meeting all lines of $\mathcal{L}\cup \{M_0,M_1\}$ namely one for each of the $(q^2+q+1)$ $(2n-3)$-spaces through the unique $(2n-4)$-space spanned by the $(2n-5)$-space $\tau$ and the unique transversal line through the point $\tau\cap \mu$ to the regulus defined by $M_0,M_1,M$. There are at most $2(n-2)(q+1)^{2n-5}$ such $(2n-5)$-spaces $\tau$.

Finally, each of the $(2n-5)$-spaces $\tau$ meeting all lines of $\mathcal{L}$ but not intersecting $M$ lie on $(q+1)^2-(q+1)$ $(2n-3)$-spaces intersecting all lines of $\mathcal{L}\cup \{M_0,M_1\}$, namely those spanned by $\tau$ and a line meeting both $M_0$ and $M_1$, but not $M$ (if the subspace $S$ would meet $M$, then $S$ would meet $\Sigma$ in $(2n-4)$-space). There are at most $(n-2)(q+1)^{2n-4}$ such $(2n-5)$-spaces $\tau$.

We find that the total number of $(2n-3)$-spaces $S$ meeting  all lines of $\mathcal{L}\cup \{M_0,M_1\}$ is at most
$$2q^{2n-3}+2(n-2)(q+1)^{2n-6}(q+1)(q^2+q+1)+ 2(n-2)(q+1)^{2n-5}(q^2+q+1)$$
$$+(n-2)(q+1)^{2n-4}(q+1)^2\leq (n-1)(q+1)^{2n-2},$$
since $q\geq 4n$.

We will now show that there is a line $N$ such that the number of $(2n-3)$-spaces meeting $N$ in a point and meeting the lines $\mathcal{L}\cup \{M_0,M_1\}$ is at most $2(n-1)(q+1)^{2n-3}$, while the number of $(2n-3)$-spaces containing $N$ and meeting the lines $\mathcal{L}\cup \{M_0,M_1\}$ is at most $2(n-1)(q+1)^{2n-4}$.

We already know that there is a set $\mathcal{S}$ of at most $(n-1)(q+1)^{2n-2}$ $(2n-3)$-spaces $S$ meeting the lines $\mathcal{L}\cup \{M_0,M_1\}$. Each of those spaces covers $\frac{q^{2n-2}-1}{q-1}$ points. Since there are $\frac{q^{2n}-1}{q-1}-(2n-2)(q+1)$ points not on the lines $\mathcal{L}$, it follows that there exists a point $P$, not on a line of $\mathcal{L}\cup \{M_0,M_1\}$ such that the number of spaces of $\mathcal{S}$ through it is at most \[\frac{(n-1)(q+1)^{2n-2}\frac{q^{2n-2}-1}{q-1}}{\frac{q^{2n}-1}{q-1}-(2n-2)(q+1)}\leq 2(n-1)(q+1)^{2n-4}.\]

Let $N$ be the spread element of $\mathcal{D}$ through $P$ and let $\mathcal{S}_P$ be the subset of spaces of $\mathcal{S}$ containing $P$. Let $x$ be the number of spaces of $\mathcal{S}_P$ containing $N$, and $y$ is the number of spaces of $\mathcal{S}_P$ not containing $N$; we have just seen that $x+y\leq 2(n-1)(q+1)^{2n-4}$ Since the lines of $\mathcal{L}\cup \{M_0,M_1\}$, as well as $N$ are elements of $\mathcal{D}$, we find that the the number of $(2n-3)$-spaces meeting all lines of $\mathcal{S}$ is given by $x+(q+1)y.$ The number of $(2n-3)$-spaces meeting all lines of $\mathcal{L}\cup \{M_0,M_1\}$ and meeting $N$ exactly in a point is $(q+1)y\leq (q+1)(x+y)\leq 2(n-1)(q+1)^{2n-4}(q+1)=2(n-1)(q+1)^{2n-3}.$

Finally, observe that, since $q>\frac{4n}{ln(2)}$, $2n(q+1)^{2n-3}>2(n-1)(q+1)^{2n-3}+2(n-1)(q+1)^{2n-4}$.

The statement now follows by induction.
\end{proof}

\begin{theorem}\label{main2}
Let $\Sigma$ be the space $\PG(2n-1,q)$, with $n\geq 2$ and $q>4n/ln(2)$. There exists a higgledy-piggledy set of size at most $2(2n-1)-2$, that is, a strong blocking set $\mathcal{S}$ in $\Sigma$ which is the union of at most $4n-4$ lines $\ell_1,\dots,\ell_{4n-4}$.
\end{theorem}
\begin{proof}
Let $\D$ denote the Desarguesian line spread of $\Sigma$, and let $\mathcal{L}=\{\ell_1,\dots,\ell_{2n-1}\}$ be a set of lines as obtained in Lemma \ref{lemma:starting lines}. We will show that we can add lines $\ell_{2n},\ell_{2n+1},\ldots,\ell_{4n-4}$ to $\mathcal{L}$ such that there is no subspace of dimension $(2n-3)$ meeting all lines $\ell_1,\ldots,\ell_{4n-4}$. 
Since it is possible that $\ell_i=\ell_j$ for some $i,j$, we will eventually obtain a set of at most $4n-4$ lines which is a higgledy-piggledy set $\mathcal{S}$: if $K$ is a hyperplane, it meets $\mathcal{S}$ in at least $2n-1$ points arising from the intersections of $K$ with the lines of $\mathcal{L}$, and by construction it is impossible that all points of $\mathcal{S}\cap K$ are contained in a hyperplane of $K$.

We double count the number of pairs 
\[
(H,Q)
\]
such that $Q\in H$ and $H$ is an $(2n-3)$-space with $H\cap \ell\neq \emptyset$ for any $\ell\in\mathcal{L}$.

By Lemma \ref{lemma:starting lines}, this number is at most
\[
2n(q+1)^{2n-3}\theta_{2n-3};
\]
where $\theta_{2n-3}=(q^{2n-2}-1)/(q-1)$ is the number of points of $H$.\\
On the other hand, the number of such pairs is also 
\[
\theta_{2n-1}M_1;
\]
where $M_1$ is the average number of $(2n-3)$-dimensional spaces $H$ meeting each of the lines $\ell_i$, $i=1,\dots,2n-1$ in at least a point and through a point $Q\in\Sigma$.

As a consequence 
\[
M_1\leq \frac{2n(q+1)^{2n-3}\theta_{2n-3}}{\theta_{2n-1}}.
\]
In particular, there exists a point $Q$ with the property that the number of $(2n-3)$-spaces through $Q$ meeting each of the lines $\ell_i$, $i=1,\dots,2n-1$ in at least a point is at most $M_1$. 
By Result \ref{stabiliser}, the unique line $\ell_{2n}$ of the spread $\mathcal{D}$ passing through $Q$ has the property that there are at most $(q+1)M_1$ $(2n-3)$-spaces meeting $\ell_{2n}$ and each of the lines $\ell_i$, $i=1,\dots,2n-1$ in at least a point. 
Note that $Q$ might be a point of one of the lines $\ell_i$, namely in such a case $\ell_{2n}\in \mathcal{L}$.
The number of pairs
\[
(H,Q)
\]
such that $H\cap \ell_i\neq\emptyset$ for any $i=1,\dots,2n$ and $Q$ belongs to $H$, is now at most
\[
(q+1)M_1\theta_{2n-3}.
\]
On the other hand, this number is also 
\[
\theta_{2n-1}M_2;
\]
where $M_2$ is the average number of $(2n-3)$-dimensional spaces $H$ through a fixed point $Q$ and meeting each of the lines $\ell_i$, $i=1,\dots,2n$ in at least a point.
As a consequence 
\[
M_2\leq \frac{(q+1)M_1\theta_{2n-3}}{\theta_{2n-1}}
\]

We can iterate the reasoning, taking each time a new line $\ell_{2n-1+i}$ contained in at most $M_i$ $(n-2)$-dimensional spaces $H$ meeting each of the previous lines, and double counting the pairs $(H,Q)$ such that $H\cap S_i\neq\emptyset$ and $Q$ belongs to $H$. 
We obtain
\[
M_i <\frac{(q+1)M_{i-1}\theta_{2n-3}}{\theta_{2n-1}}.
\]
Combining with $M_1<\frac{2n(q+1)^{2n-3}\theta_{2n-3}}{\theta_{2n-1}}$, we obtain 
\[
M_i< \frac{2n(q+1)^{2n-3+i-1}\theta_{2n-3}^i}{\theta_{2n-1}^i}.
\]
In particular, 
\[
M_{2n-3}< \frac{2n(q+1)^{4n-7}\theta_{2n-3}^{2n-3}}{\theta_{2n-1}^{2n-3}}
=2n(q+1)^{4n-7}\left(\frac{q^{2n-2}-1}{q^{2n}-1}\right)^{2n-3}<2n(q+1)^{4n-7}\left(\frac{1}{q^{2}}\right)^{2n-3}
\]

Observe that 
\[
(q+1)^{4n-7}<q^{4n-7}+(4n-6)q^{4n-8},
\]
holds for any $q$ such that 
\[
q>\frac{1}{e^{ln(2)/(4n-7)}-1}.
\]

Indeed $(q+1)^{4n-7}<q^{4n-7}+(4n-6)q^{4n-8}$ is equivalent to $(1+1/q)^{4n-7}<1+\frac{4n-6}{q}$. Assuming $q>4n-6$, the right hand side of this inequality is at most $2$. Isolating $q$, we obtain the claim.
Therefore, 
\[
M_{2n-3}<2n\left(\frac{1}{q}+\frac{4n-6}{q^2}\right).
\]

Finally, $2n\left(\frac{1}{q}+\frac{4n-6}{q^2}\right)<1$ is equivalent to $q^2-2nq-4n+6>0$, which is always true under our hypothesis.
It can be checked that if $q>4n/ln(2)$ and $n\geq 2$, then $q>\frac{1}{e^{ln(2)/(4n-7)}-1}$ as well. 
Since $M_{2n-3}<1$, there must be a point not contained in any $(2n-3)$-subspace meeting all the lines $\ell_i$, $i=1,\dots,4n-5$. Again using Result \ref{stabiliser}, we see that the unique line $\ell_{4n-4}$ through that point is such that there is no $(2n-3)$-subspace meeting all the lines $\ell_i$, $i=1,\dots,4n-4$.
\end{proof}

\section*{Acknowledgments}
The first author acknowledges the support by the Irish Research Council, grant n. GOIPD/2022/307 and by the Italian institute for high mathematics INdAM, as ``titolare di una borsa per l’estero dell’Istituto Nazionale di Alta Matematica".

\bibliography{Bibliography}
\end{document}